\documentclass{article}

\usepackage{amsmath,amssymb,amsthm,enumerate,mathrsfs}
\usepackage[numbers]{natbib}
\usepackage{booktabs}
\usepackage{adjustbox}
\usepackage{subfigure,multirow}
\usepackage{appendix}
\usepackage{authblk}
\usepackage{geometry,xcolor}
\usepackage[ruled,linesnumbered]{algorithm2e}
\usepackage{algpseudocode}%
\usepackage{longtable}
\usepackage{caption}
\usepackage{lipsum}
\usepackage{footnote}
\usepackage[misc]{ifsym}

\newcommand{\T}{\text{T}}
\newcommand{\x}{\mathbf{x}}
\newcommand{\g}{\mathbf{g}}
\newcommand{\bb}{\mathbf{b}}

\newcommand{\s}{\mathbf{s}}
\newcommand{\y}{\mathbf{y}}
\newcommand{\xo}{\mathbf{0}}
\newcommand{\BB}{Barzilai-Borwein }
\newcommand{\BBB}{Barzilai and Borwein }

\renewcommand{\thefootnote}{\fnsymbol{footnote}}

\theoremstyle{plain}
\newtheorem{theorem}{Theorem}
\newtheorem{lemma}[theorem]{Lemma}
\newtheorem{proposition}[theorem]{Proposition}

\theoremstyle{remark}
\newtheorem{definition}{Definition}
\newtheorem{remark}{Remark}

\newcommand\blfootnote[1]{%
	\begingroup 
	\renewcommand\thefootnote{}\footnote{#1} 
	\addtocounter{footnote}{-1}%
	\endgroup 
}

\title{\bf On Convergence of Regularized Barzilai-Borwein Method}

\author{Xin Xu}
\date{}

\begin{document}
\maketitle

\blfootnote{\Letter Xin Xu \par \; xustonexin@163.com} 
\footnotetext{School of Mathematics, Southwestern University of Finance and Economics, Chengdu 611130, Sichuan, China}

\begin{abstract}
	The regularized Barzilai-Borwein (RBB) method represents a promising gradient-based optimization algorithm. In this paper, by splitting the gradient into two parts and analyzing the dynamical system of difference equations governing the
	ratio of their magnitudes, we establish that the RBB method achieves R-linear convergence for strongly convex quadratic
	functions of arbitrary dimensions. Specifically, for the two-dimensional case, we provide a concise proof demonstrating that the method exhibits at least R-linear convergence. We propose a simple yet effective adaptive regularization parameter
	scheme to further improve its performance. A typical numerical example verifies the effectiveness of this scheme.
	
	
\end{abstract}

\section{Introduction}\label{sec1}

Consider the problem of minimizing a strictly convex quadratic as follows 
\begin{equation}\label{equ:objective}
\min_{\mathbf{x}\in\mathbb{R}^{n}} f(\x)=\frac{1}{2}\x^{\T}A\x-\mathbf{b}^{\T}\x,
\end{equation}
where $A\in\mathbb{R}^{n\times n}$ is a real symmetric positive definite matrix and $\bb\in\mathbb{R}^{n}$. For solving \eqref{equ:objective}, gradient-like descent methods have the following iterative form 
\begin{equation}
\x_{k+1}=\x_{k}+\alpha_{k}^{-1}(-\g_{k}),
\end{equation}  
where $\x_{k}$ is the $k$th iterate, $\g_{k}=\nabla f(\x_{k})$ is the gradient of $f$ at $\x_{k}$, $\alpha_{k}^{-1}$ is the step size. Different gradient descent methods would have different rules for determining the $\alpha_{k}$. The classical steepest descent (SD) method dates back to Cauchy \cite{Cauchy2009Methodegeneralepour},  which determines its step size by the so-called exact line search
\begin{equation*}\label{SDG}
\alpha_{k}^{SD}=\mathop{\text{argmin}}_{\alpha_{k}>0} f\big(\mathbf{x}_{k}+\alpha_{k}^{-1}(-\mathbf{g}_{k})\big)=\frac{\g_{k}^{\T}A\g_{k}}{\g_{k}^{\T}\g_{k}}.
\end{equation*}
Nevertheless, when the SD method is applied to the any-dimensional quadratics, the search directions it generates usually tend to be two orthogonal directions, i.e., the ``zigzagging" phenomenon occurs (see Akaike \cite{Akaike1959successivetransformationprobability}). In particular, when $A$ is badly conditioned, this phenomenon will seriously reduce the convergence rate of the SD method  \cite{Sun2006OptimizationTheoryMethods,JorgeNocedal2006NumericalOptimization}.

In 1988, \BBB (BB) \cite{Barzilai1988TwoPointStep} proposed two novel $\alpha_{k}$ as follows
\begin{equation}\label{BB steps}
\alpha_{k}^{BB1}=\frac{\mathbf{s}_{k-1}^{\T}\mathbf{y}_{k-1}}{\mathbf{s}_{k-1}^{\T}\mathbf{s}_{k-1}}=\frac{\g_{k-1}^{\T}A\g_{k-1}}{\g_{k-1}^{\T}\g_{k-1}}\quad\text{and}\quad \alpha_{k}^{BB2}=\frac{\mathbf{y}_{k-1}^{\T}\mathbf{y}_{k-1}}{\mathbf{s}_{k-1}^{\T}\mathbf{y}_{k-1}}=\frac{\g_{k-1}^{\T}A^2\g_{k-1}}{\g_{k-1}^{\T}A\g_{k-1}},
\end{equation} 
which are the solutions to the following least squares problems
\begin{equation*}\label{LS}
\min_{\alpha\in\mathbb{R}} \|\alpha \mathbf{s}_{k-1} - \mathbf{y}_{k-1}\|_{2}^{2}\quad\text{and}\quad \min_{\alpha\in \mathbb{R}} \|\mathbf{s}_{k-1} -  \alpha^{-1}\mathbf{y}_{k-1}\|_{2}^{2},
\end{equation*} 
respectively, where $\mathbf{s}_{k-1}=\mathbf{x}_{k}-\mathbf{x}_{k-1}$ and $\mathbf{y}_{k-1}=\mathbf{g}_{k}-\mathbf{g}_{k-1}$. By the Cauchy-Schwarz inequality, one has $\alpha_{k}^{BB1}\le\alpha_{k}^{BB2}$. In the least squares sense, the BB method approximates the Hessian of $f(\mathbf{x}_{k})$ using $\alpha_{k}^{BB1}I$ or $\alpha_{k}^{BB2}I$, where $I$ is the identity matrix. Therefore, the BB method incorporates quasi-Newton property to gradient descent method by the choice of step size. In the seminal paper \cite{Barzilai1988TwoPointStep}, the authors proved that the BB method converges \text{R-superlinearly} to the global minimizer when $n=2$. In any-dimensional problem, it is still globally convergent \cite{Raydan1993BarzilaiBorweinchoice} and the convergence is R-linear  \cite{Dai2002Rlinearconvergence}. A further analysis on the asymptotic behaviors of BB-like methods can be found in \cite{Dai2005AsymptoticBehaviourSome}.

The BB method has attracted considerable attention in computational mathematics and optimization communities owing to its low computational complexity and remarkable empirical performance. To mention just a few representative works, Raydan \cite{Raydan1997BarzilaiBorweinGradient} developed an efficient global BB algorithm for unconstrained optimization by incorporating the conventional non-monotonic line search proposed by Grippo et al. \cite{Grippo1986NonmonotoneLineSearch}, thereby facilitating the practical application of the BB methodology. The adaptive alternating BB step-sizes proposed by Zhou et al. \cite{Zhou2006GradientMethodsAdaptive} effectively revealed the spectral scanning property inherent in BB step-sizes, thereby providing valuable insights for numerous subsequent improvements to BB-type methods. Furthermore, it is particularly noteworthy that Xu \cite{Xu2025IBB} introduced an interpolated least-squares model that elegantly unifies the two original BB step sizes. Additional refinements and applications of the BB method can be found in \cite{Molina1996PreconditionedBarzilaiBorwein, Friedlander1998GradientMethodRetards, Dai2005ProjectedBarzilaiBorweina, Dai2005Newalgorithmssingly, Huang2021EquippingBarzilaiBorwein, Crisci2020SpectralPropertiesBarzilai, An2020Numericalconstructionspherical, Jalilian2023}.

The superior performance of the BB method demonstrates that the behavior of the SD algorithm cannot be solely attributed to its search direction, and that designing an efficient step size is a critical factor. Recently, An and Xu \cite{An2024RegularizedBarzilaiBorwein} proposed a regularized \BB (RBB) method, which yields a scalar $\alpha_{k}^{RBB}$ defined as follows
\begin{equation}\label{equ:RBB}
\alpha_{k}^{RBB}=\frac{\mathbf{s}_{k-1}^{\T}\mathbf{y}_{k-1}+\tau_{k} \mathbf{y}_{k-1}^{\T} \mathbf{y}_{k-1}}{\mathbf{s}_{k-1}^{\T}\mathbf{s}_{k-1}+\tau_{k} \mathbf{s}_{k-1}^{\T} \mathbf{y}_{k-1}},
\end{equation}
which is the solution to the following regularized least-squares problem 
\begin{equation*}\label{equ:LSR}
\min_{\alpha\in\mathbb{R}} \Big\{\Vert \alpha \s_{k-1}-\y_{k-1}\Vert_{2}^{2} + \tau_{k}\Vert \alpha\sqrt{A}\s_{k-1} - \sqrt{A}\y_{k-1}\Vert_{2}^{2} \Big\},
\end{equation*}
where $\tau_{k}\ge0$ is the regularization parameter. Note that when $\tau_{k}=0$, $\alpha_{k}^{RBB}$ reduces to the  $\alpha_{k}^{BB1}$. Evidently, $\alpha_{k}^{RBB}$ can be viewed as a modification of $\alpha_{k}^{BB1}$ through the incorporation of a regularization term, while still preserving the quasi-Newton property. Although the BB method is a non-monotonic gradient descent algorithm, such non-monotonicity does not compromise its overall convergence. Nevertheless, the oscillation in the gradient norm induced by this behavior may adversely affect the convergence speed, particularly for the BB1 method which tends to produce longer step sizes. The design of the RBB step size aims to enhance the stability of the BB1 method by introducing a regularization term, thereby improving its overall performance.

Through a suitable regularization parameter scheme, the numerical results of \cite{An2024RegularizedBarzilaiBorwein,Zhao2024conicregularizedBarzilaiBorwein} show that the RBB method is a very promising gradient descent method. Nevertheless, \cite{An2024RegularizedBarzilaiBorwein} does not provide a rigorous convergence analysis of the RBB method. To address this gap, this paper quadratic problem \eqref{equ:objective}, we formally establishes the convergence properties of the RBB method for the quadratic optimization problem \eqref{equ:objective}.

This work makes the following three contributions:
\begin{itemize}
	\item By splitting the gradient into two components and analyzing the dynamical system of difference equations governing the ratio of their magnitudes, we establish the R-linear convergence of the RBB method for strongly convex quadratic functions of arbitrary dimensions. Furthermore, we demonstrate that the convergence rate of the RBB method is determined by the convergence rate of the gradient component corresponding to the smallest eigenvalue of the Hessian matrix, while its stability depends on the convergence behavior of the other gradient component.
	\item We analyze several special cases of the RBB method in the two-dimensional case and provide a more succinct and illustrative proof demonstrating that the method achieves at least R-linear convergence.
	\item Based on the analysis framework established in this paper, we uniformly adopt the perspective of difference equation dynamics to analyze the convergence of the gradient descent methods, encompassing both the gradient descent method with a delayed multi-step strategy and the classical steepest descent method.
\end{itemize}

This paper is organized as follows. In Section \ref{sec:n-dimensional}, for quadratic minimization problems of arbitrary dimensions, we first present some fundamental settings, then conduct a detailed analysis of the dynamical properties of the second-order difference equation generated by the RBB method for the sequence  $\{r_k\}$, and thereby establish the convergence of the RBB method. Finally, we examine the behavioral characteristics of the RBB method and the BB1 method in the two-dimensional case. In Section \ref{sec:discussions}, we extend the analytical framework developed in this paper to more general gradient descent schemes, including those with a delayed multi-step strategy and the steepest descent method. In Section \ref{sec:parameter}, we establish an effective adaptive regularization parameter scheme. In Section \ref{sec:NumerEx}, we preliminarily conduct numerical experiments.  

\section{Convergence analysis}\label{sec:n-dimensional}
In this section, we first establish the $R$-linear convergence of the RBB method for  any-dimensional quadratic function \eqref{equ:objective}. Then, we focus on several special cases of $\alpha_{k}^{RBB}$ in two-dimensional problems.
\subsection{Preliminary settings}\label{sec:2.0}
Without loss of generality, we consider a matrix $A$ of the following form 
\begin{equation*}
A=\begin{bmatrix}
A_{1} & \xo\\
\xo & 1
\end{bmatrix},
\end{equation*}
where the sub-matrix 
$$A_{1}=\text{diag}(a_{1},a_{2},\ldots,a_{n-1}),$$
where $\lambda:=a_{1}\ge a_{2}\ge\ldots\ge a_{n-1}:=\gamma>1$, $\lambda>2$, and $n\ge2$. 

We denote $\g_{k-1}=\big((\g_{k-1}^{1})^{\T}, \ \g_{k-1}^{(2)}\big)^{\T}$ with $\g_{k-1}^{1}=\big(\g_{k-1}^{(1)}, \g_{k-1}^{(2)},\ldots,\g_{k-1}^{(n-1)}\big)^{\T}\in\mathbb{R}^{n-1}(\neq\xo)$,  $\g_{k-1}^{(n)}\neq 0$, and let
\begin{equation}\label{equ:nepsilon}
r_{k-1}=\frac{\|\g_{k-1}^{1}\|_{2}^{2}}{(\g_{k-1}^{(n)})^{2}}>0.
\end{equation}
Then we have 
\begin{align}\label{equ:normgk}
\begin{split}
\|\g_{k-1}\|_{2}^{2}&=\|\g_{k-1}^{1}\|_{2}^{2}+(\g_{k-1}^{(n)})^2\\
&=(\g_{k-1}^{(n)})^2(r_{k-1}+1).
\end{split}	
\end{align}

According to the setting of $A_{1}$ and the properties of the Rayleigh quotient, it follows that $\gamma\le\frac{(\g_{k-1}^{1})^{\T}A_{1}\g_{k-1}^{1}}{\|\g_{k-1}^{1}\|_{2}^{2}}\le\lambda$ and  $\gamma\le\frac{(\g_{k}^{1})^{\T}A_{1}\g_{k}^{1}}{\|\g_{k}^{1}\|_{2}^{2}}\le\lambda$. Consequently, there exist $\eta_{k-1}\in[\frac{\gamma}{\lambda},1]$ and $\eta_{k}\in[\frac{\gamma}{\lambda},1]$ such that
\begin{equation}\label{R1}
(\g_{k-1}^{1})^{\T}A_{1}\g_{k-1}^{1}=\eta_{k-1}\lambda\|\g_{k-1}^{1}\|_{2}^{2},\quad\text{and}\quad(\g_{k}^{1})^{\T}A_{1}\g_{k}^{1}=\eta_{k}\lambda\|\g_{k}^{1}\|_{2}^{2}.
\end{equation} 
Similarly, we know that there exist $\bar{\eta}_{k-1}\in[\frac{\gamma}{\lambda},1]$ and $\bar{\eta}_{k}\in[\frac{\gamma}{\lambda},1]$ such that
\begin{equation}\label{R2}
(\g_{k-1}^{1})^{\T}A_{1}^2\g_{k-1}^{1}=\bar{\eta}_{k-1}\eta_{k-1}\lambda^2\|\g_{k-1}^{1}\|_{2}^{2},\quad\text{and}\quad(\g_{k}^{1})^{\T}A_{1}^2\g_{k}^{1}=\bar{\eta}_{k}\eta_{k}\lambda^2\|\g_{k}^{1}\|_{2}^{2}.	
\end{equation}
Let $w_{i}=\frac{(\g_{k-1}^{(i)})^{2}}{\sum_{j=1}^{n-1}(\g_{k-1}^{(j)})^{2}}$, for $i=1,2,\ldots,n-1$, so that $\sum_{i=1}^{n-1}w_{i}=1$ and $w_{i}\ge0$. Let $a=\sum_{i=1}^{n-1}w_{i}a_{i}$ and $b=\sum_{i=1}^{n-1}w_{i}a_{i}^{2}$. By Jensen's inequality and since the function $f(x)=x^2$ is convex, we have
\begin{equation}\label{Jensen}
b=\sum_{i=1}^{n-1}w_{i}a_{i}^{2}\ge(\sum_{i=1}^{n-1}w_{i}a_{i})^2=a^2.
\end{equation} 
Equality holds if and only if all $a_{i}$ are equal, i.e., $a_{1}=a_{2}=\ldots=a_{n-1}$. It follows from \eqref{R1}, \eqref{R2} and \eqref{Jensen} that
\begin{equation}\label{bareta}
\bar{\eta}_{k-1}\ge\eta_{k-1},\quad\text{and}\quad\bar{\eta}_{k}\ge\eta_{k}
\end{equation} 
holds. 


Since $\s_{k-1}=-(\alpha_{k-1}^{RBB})^{-1}\g_{k-1}$ and $\y_{k-1}=-(\alpha_{k-1}^{RBB})^{-1}A\g_{k-1}$, according the definition of $\alpha_{k}^{RBB}$ \eqref{equ:RBB} and the results \eqref{R1}, \eqref{R2}, \eqref{bareta}, we have 
\begin{align}\label{equ:nnRBB}
\begin{aligned}
\alpha_{k}^{RBB}&=\frac{\g_{k-1}^{\T}A\g_{k-1}+\tau_{k}\g_{k-1}^{\T}A^2\g_{k-1}}{\g_{k-1}^{\T}\g_{k-1}+\tau_{k}\g_{k-1}^{\T}A\g_{k-1}}\\
&=\frac{(\g_{k-1}^{1})^{\T}A_{1}\g_{k-1}^{1}+(\g_{k-1}^{(n)})^2+\tau_{k}\Big[(\g_{k-1}^{1})^{\T}A_{1}^2\g_{k-1}^{1}+(\g_{k-1}^{(n)})^2\Big]}{\|\g_{k-1}^{1}\|_{2}^{2}+(\g_{k-1}^{(n)})^2+\tau_{k}\Big[(\g_{k-1}^{1})^{\T}A_{1}\g_{k-1}^{1}+(\g_{k-1}^{(n)})^2\Big]}\\
&=\frac{\eta_{k-1}\lambda\|\g_{k-1}^{1}\|_{2}^{2}+(\g_{k-1}^{(n)})^2+\tau_{k}\Big[\bar{\eta}_{k-1}\eta_{k-1}\lambda^2\|\g_{k-1}^{1}\|_{2}^{2}+(\g_{k-1}^{(n)})^2\Big]}{\|\g_{k-1}^{1}\|_{2}^{2}+(\g_{k-1}^{(n)})^2+\tau_{k}\Big[\eta_{k-1}\lambda\|\g_{k-1}^{1}\|_{2}^{2}+(\g_{k-1}^{(n)})^2\Big]}\\
&=\frac{1+r_{k-1}\eta_{k-1}\lambda+\tau_{k}[r_{k-1}\bar{\eta}_{k-1}\eta_{k-1}\lambda^2+1]}{1+r_{k-1}+\tau_{k}[r_{k-1}\eta_{k-1}\lambda+1]},
\end{aligned}
\end{align}
with $\tau_{k}\ge0$. 

Because 
\begin{align}\label{hhk}
\g_{k+1}=\big(I-(\alpha_{k}^{RBB})^{-1}A\big)\g_{k},
\end{align}
from \eqref{R1}, \eqref{R2}, and  \eqref{equ:nnRBB}, 
we have  
\begin{align}\label{equ:gkk11}
\begin{split}
\|\g_{k+1}^{1}\|_{2}^{2}&=\|\big(I-(\alpha_{k}^{RBB})^{-1}A_{1}\big)\g_{k}^{1}\|_{2}^{2}\\
&=(\g_{k}^{1})^{\T}\big(I-(\alpha_{k}^{RBB})^{-1}A_{1}\big)^{\T}\big(I-(\alpha_{k}^{RBB})^{-1}A_{1}\big)\g_{k}^{1}\\
&=(\g_{k}^{1})^{\T}\big(I-\frac{2}{\alpha_{k}^{RBB}}A_{1}+\frac{1}{(\alpha_{k}^{RBB})^2}A_{1}^{2}\big)\g_{k}^{1}\\
&=\frac{(\alpha_{k}^{RBB})^{2}-2\alpha_{k}^{RBB}\eta_{k}\lambda+\bar{\eta}_{k}\eta_{k}\lambda^{2}}{(\alpha_{k}^{RBB})^{2}}\|\g_{k}^{1}\|_{2}^{2}\\
&=\beta_{k}\|\g_{k}^{1}\|_{2}^{2},
\end{split}	
\end{align}
\begin{align}\label{equ:gkk22}
\begin{split}
(\g_{k+1}^{(n)})^2&=\big(1-(\alpha_{k}^{RBB})^{-1}\big)^2(\g_{k}^{(n)})^{2}\\
&=\xi_{k}^2(\g_{k}^{(n)})^2,
\end{split}
\end{align}
where
	\begin{align}\label{equ:betak}
	\begin{split}
	\beta_{k}=\frac{\big(P+\tau_{k}C\big)^2+\big(Q+\tau_{k}E\big)^2\eta_{k}(\bar{\eta}_{k}-\eta_{k})\lambda^2}{\big[1+r_{k-1}\eta_{k-1}\lambda+\tau_{k}(1+r_{k-1}\bar{\eta}_{k-1}\eta_{k-1}\lambda^2)\big]^2}
	\end{split}
	\end{align}
	with $P=1-\eta_{k}\lambda+r_{k-1}\lambda(\eta_{k-1}-\eta_{k})$, $C=1-\eta_{k}\lambda+r_{k-1}\eta_{k-1}\lambda^2(\bar{\eta}_{k-1}-\eta_{k})$, $Q=1+r_{k-1}>1$,
	$E=r_{k-1}\eta_{k-1}\lambda+1>1$,
and
\begin{equation}\label{equ:xik}
\xi_{k}=\frac{r_{k-1}(\eta_{k-1}\lambda-1)+\tau_{k}r_{k-1}\eta_{k-1}\lambda(\bar{\eta}_{k-1}\lambda-1)}{1+r_{k-1}\eta_{k-1}\lambda+\tau_{k}(1+r_{k-1}\bar{\eta}_{k-1}\eta_{k-1}\lambda^2)}\in(0,1).	
\end{equation} 
From \eqref{equ:nepsilon}, \eqref{equ:gkk11}, and \eqref{equ:gkk22}, we have 
\begin{equation}\label{equ:nrecurion}
	\frac{r_{k-1}^{2}r_{k+1}}{r_{k}}=h_{k-1},
\end{equation}
where
\begin{small}
	\begin{align}\label{equ:hkkk}
	\begin{aligned}
	h_{k-1}=\frac{\big(P+\tau_{k}C\big)^2+\big(Q+\tau_{k}E\big)^2\eta_{k}(\bar{\eta}_{k}-\eta_{k})\lambda^2}{\big[(\eta_{k-1}\lambda-1)+\tau_{k}\eta_{k-1}\lambda(\bar{\eta}_{k-1}\lambda-1)\big]^2}.
	\end{aligned}
	\end{align}
\end{small}
Essentially, equation \eqref{equ:nrecurion} defines a second-order non-homogeneous constant coefficient difference equation with respect to $r_{k}$  \cite{Elaydi2005IntroductionDifferenceEquations,Kelley1991DifferenceEquationsIntroduction}. Further, we can rewritten the value of $h_{k-1}$ as follows
\begin{equation*}
h_{k-1}=\delta_{k}^2\big[(\alpha_{k}^{RBB}-\eta_{k}\lambda)^2+\eta_{k}(\bar{\eta}_{k}-\eta_{k})\lambda^2\big],
\end{equation*}
where 
$$\delta_{k}=\frac{1+r_{k-1}+\tau_{k}(r_{k-1}\eta_{k-1}\lambda+1)}{(\eta_{k-1}\lambda-1)+\tau_{k}\eta_{k-1}\lambda(\bar{\eta}_{k-1}\lambda-1)}.$$
\begin{proposition}\label{prop:deltak}
	$\delta_{k}$ is monotonically decreasing with respect to $\tau_{k}$ and $\delta_{k}\in\Big[\frac{r_{k-1}\eta_{k-1}\lambda+1}{\eta_{k-1}\lambda(\bar{\eta}_{k-1}\lambda-1)},\frac{1+r_{k-1}}{\eta_{k-1}\lambda-1}\Big]$.
\end{proposition}
\begin{proof}
Directly computing the derivative of $\delta_{k}$ with respect to $\tau_{k}$ yields that 
\begin{align*}
	\begin{aligned}
	\frac{d(\delta_{k})}{d(\tau_{k})}&=\frac{r_{k-1}\eta_{k-1}^2\lambda^2-(r_{k-1}+1)\eta_{k-1}\bar{\eta}_{k-1}\lambda^2+2\eta_{k-1}\lambda-1}{\big[(\eta_{k-1}\lambda-1)+\tau_{k}\eta_{k-1}\lambda(\bar{\eta}_{k-1}\lambda-1)\big]^2}\\
	&\le\frac{-(\eta_{k-1}\lambda-1)^2}{\big[(\eta_{k-1}\lambda-1)+\tau_{k}\eta_{k-1}\lambda(\bar{\eta}_{k-1}\lambda-1)\big]^2}\\
	&<0
	\end{aligned}
\end{align*}
holds. Therefore, the maximum and minimum values of $\delta_{k}$ are attained at $\tau_{k}=0$ and $\infty$, respectively. This completes the proof.
\end{proof}
Combine Proposition \ref{prop:deltak} and the fact of  $\big[(\alpha_{k+1}^{RBB}-\eta_{k}\lambda)^2+\eta_{k}(\bar{\eta}_{k}-\eta_{k})\lambda^2\big]<2\lambda^2.$
Then, we have
\begin{equation}\label{equ:hk}
h_{k-1}< 2\frac{(1+r_{k-1})^2}{(\eta_{k-1}\lambda-1)^2}\lambda^2,
\end{equation} 
which implies that $h_{k-1}$ is bounded. This conclusion is crucial, as it demonstrates that the non-homogeneous term in Equation \eqref{equ:nrecurion} does not exert an overwhelming influence on the general solution corresponding to the homogeneous equation. The results of Proposition \ref{prop:deltak} indicates that increasing the regularization parameter can reduce the value of the non-homogeneous term, thereby effectively diminishing its influence on the homogeneous general solution.

\begin{theorem}\label{prop:increasing}
	In \eqref{equ:xik}, the $\xi_{k}$ is monotonically increasing with respect to the regularization parameter $\tau_{k}$, and we have 
	\begin{equation}\label{equ:fanwei}
	\xi_{k}\in(0,1-\frac{1}{\lambda}).
	\end{equation}	
\end{theorem}
\begin{proof}
	Since $r_{k-1}>0$, it suffices to examine the monotonicity of $\hat{\xi}_{k}$ with respect to $\tau_{k}$, where 
	\begin{equation*}
		\hat{\xi}_{k}=\frac{(\eta_{k-1}\lambda-1)+\tau_{k}\eta_{k-1}\lambda(\bar{\eta}_{k-1}\lambda-1)}{1+r_{k-1}\eta_{k-1}\lambda+\tau_{k}(1+r_{k-1}\bar{\eta}_{k-1}\eta_{k-1}\lambda^2)}.
	\end{equation*} 
Directly computing the derivative of $\hat{\xi}_{k}$ with respect to $\tau_{k}$ yields that 
\begin{align}
	\begin{aligned}
	\frac{d(\hat{\xi}_{k})}{d(\tau_{k})}&=\frac{(1+r_{k-1})\eta_{k-1}\bar{\eta}_{k-1}\lambda^2-r_{k-1}\eta_{k-1}^2\lambda^2-2\eta_{k-1}\lambda+1}{\big[1+r_{k-1}\eta_{k-1}\lambda+\tau_{k}(1+r_{k-1}\bar{\eta}_{k-1}\eta_{k-1}\lambda^2)\big]^2}\\
	&\ge \frac{(1+r_{k-1})\eta_{k-1}^2\lambda^2-r_{k-1}\eta_{k-1}^2\lambda^2-2\eta_{k-1}\lambda+1}{\big[1+r_{k-1}\eta_{k-1}\lambda+\tau_{k}(1+r_{k-1}\bar{\eta}_{k-1}\eta_{k-1}\lambda^2)\big]^2}\\
	&=\frac{(\eta_{k-1}\lambda-1)^2}{\big[1+r_{k-1}\eta_{k-1}\lambda+\tau_{k}(1+r_{k-1}\bar{\eta}_{k-1}\eta_{k-1}\lambda^2)\big]^2}\\
	&>0.
	\end{aligned}
\end{align}	
Therefore, $\hat{\xi}_{k}$ is monotonically increasing with respect to the regularization parameter $\tau_{k}$, which implies that the value of $\xi_{k}$ is also monotonically increasing with respect to $\tau_{k}$. Furthermore, given that
	\begin{equation}\label{equ:jie}
	\lim\limits_{\tau_{k}\to\infty}\xi_{k}=\frac{r_{k-1}\eta_{k-1}\lambda(\bar{\eta}_{k-1}\lambda-1)}{1+r_{k-1}\bar{\eta}_{k-1}\eta_{k-1}\lambda^2}<\frac{r_{k-1}\eta_{k-1}\bar{\eta}_{k-1}\lambda^2-r_{k-1}\eta_{k-1}\lambda}{r_{k-1}\bar{\eta}_{k-1}\eta_{k-1}\lambda^2}=1-\frac{1}{\bar{\eta}_{k-1}\lambda},
	\end{equation}
and  $\bar{\eta}_{k-1}\in[\frac{\gamma}{\lambda},1]$, we have
	\begin{equation}\label{equ:uplim}
	1-\frac{1}{\bar{\eta}_{k-1}\lambda}\in\big[1-\frac{1}{\gamma}, 1-\frac{1}{\lambda}\big].
	\end{equation}
	Based on the monotonicity of $\xi_{k}$ with respect to $\tau_{k}$ and Equations \eqref{equ:jie} and \eqref{equ:uplim}, we know that 
	\begin{equation*}\label{equ:condition}
	\xi_{k}\in\big(0, 1-\frac{1}{\lambda}\big)
	\end{equation*}
	holds for all $k\ge 1$. We complete the proof.
	
\end{proof}

It follows from \eqref{equ:gkk22} that
\begin{equation*}\label{equ:ratio}
|\g_{k+1}^{(2)}|<|\g_{k}^{(2)}|
\end{equation*}
with ratio $\xi_{k}$.
From \eqref{equ:nepsilon}, \eqref{equ:normgk}, \eqref{equ:gkk22}, and \eqref{equ:xik}, we know that 
\begin{equation}\label{rtio}
\frac{\|\g_{k+1}\|_{2}^{2}}{\|\g_{k}\|_{2}^{2}}=\xi_{k}^2\frac{r_{k+1}+1}{r_{k}+1}
\end{equation}
holds. From \eqref{rtio}, it follows that if  $\frac{r_{k+1}}{r_{k}}\le1$ for some $k\ge1$, then $\|\g_{k+1}\|_{2}$ is decreasing compared to $\|\g_{k}\|_{2}$; otherwise, $\|\g_{k+1}\|_{2}$ may increase. Further, it follows from \eqref{rtio} that the following relation  
\begin{equation}\label{prod}
\frac{\|\g_{k+N}\|_{2}^{2}}{\|\g_{k}\|_{2}^{2}}=\frac{\|\g_{k+1}\|_{2}^{2}}{\|\g_{k}\|_{2}^{2}}\frac{\|\g_{k+2}\|_{2}^{2}}{\|\g_{k+1}\|_{2}^{2}}\ldots\frac{\|\g_{k+N}\|_{2}^{2}}{\|\g_{k+N-1}\|_{2}^{2}}=\Big(\prod_{j=0}^{N-1}\xi_{k+j}^2\Big)\frac{r_{k+N}+1}{r_{k}+1}
\end{equation}
holds. From \eqref{prod} and $\xi_{k}\in(0,1-\frac{1}{\lambda})$, it follows that if it can be proven that there always exists an integer $N\ge1$ such that $\frac{r_{k+N}}{r_{k}}\le1$ for $k\ge1$, then the convergence of $\{\|\g_{k}\|_{2}\}$ can be established. To this end, we focus on studying the properties of difference equation \eqref{equ:nrecurion}. 

\subsection{Properties of solution \eqref{equ:nrecurion}}\label{sec:PS}
As mentioned in the introduction, in the BB-like methods, the magnitude of the gradient typically exhibits a non-monotonic oscillatory descent pattern. Usually in literature oscillatory behavior is studied only around $0$, whereas in applications mostly we need oscillation around some other point, e.g. at a stationary point $r^{*}$.

\begin{definition}\cite[p.3]{Agarwal2000Oscillation}
	The sequence $\{r_{k}\}_{k=0}^{\infty}$ is said to be strictly oscillatory around $0$, if for every $k\in\mathbb{N}:=\{1,2,\cdots\}$ there exists a $n\in\mathbb{N}$ such that $(r_{k})(r_{k+n})<0$. Further, $\{r_{k}\}$ is said to be non-oscillatory if it is eventually of constant sign.
\end{definition}

We perform the logarithmic transformation on Equation \eqref{equ:nrecurion}. Let
\begin{equation}\label{equ:log}
	y_{k}=\ln(r_{k})\quad\text{and}\quad m_{k}=\ln(h_{k}).
\end{equation}  
Then we obtain a second-order linear non-homogeneous difference equation with constant coefficient as follows
\begin{equation}\label{equ:secondDF}
	y_{k+2}-y_{k+1}+2y_{k}=m_{k}.
\end{equation}
Divide both sides of equation \eqref{equ:secondDF} by $(\sqrt{2})^{k+2}$, and let 
\begin{equation}\label{equ:ukk}
	u_{k}=\frac{y_{k}}{(\sqrt{2})^{k}}.
\end{equation}
Then we have
\begin{equation}\label{equ:secondDF2}
u_{k+2}-\frac{1}{\sqrt{2}}u_{k+1}+u_{k}=\widehat{m}_{k},
\end{equation}
where 
\begin{equation}\label{equ:mk}
	\widehat{m}_{k}=\frac{m_{k}}{(\sqrt{2})^{k+2}}.
\end{equation}
We explore the properties of equation \eqref{equ:secondDF2}. Let the solution of  \eqref{equ:secondDF2} be
\begin{equation}\label{equ:solution}
	u_{k}=u_{k}^{(g)}+u_{k}^{(p)},
\end{equation}
where $u_{k}^{(g)}$ is the general solution of the homogeneous equation and $u_{k}^{(p)}$ is the particular solution corresponding to the non-homogeneous term. 

By solving the characteristic equation of \eqref{equ:secondDF2} as follows
\begin{equation*}
q^2-\frac{1}{\sqrt{2}}q+1=0,
\end{equation*}
it can be known that its two solutions are the following complex numbers
\begin{equation*}
	q=e^{\pm i\theta},
\end{equation*}
where $\theta=\arccos{(\frac{1}{\sqrt{8}})}$, and their magnitudes are $1$. Therefore, the homogeneous general solution of equation \eqref{equ:secondDF2} can be expressed as 
\begin{equation}\label{equ:general}
	u_{k}^{(g)}=c_{1}\cos{(k\theta)}+c_{2}\sin{(k\theta)},
\end{equation}
where the values of $c_{1}$ and $c_{2}$ are determined by the initial iterates. 

The next task is to find the particular solution for the non-homogeneous term. We solve it using the variation of constants method. Based on the form of homogeneous solution \eqref{equ:general}, we seek a particular solution of the following form
\begin{equation}\label{equ:ukp}
	u_{k}^{(p)}=C_{k}\cos{(k\theta)}+D_{k}\sin{(k\theta)},
\end{equation}
where $C_{k}$ and $D_{k}$ are the undetermined sequences. The variation of constants method requires that equation \eqref{equ:ukp} satisfies the following constraint
\begin{equation}\label{equ:DD}
\begin{cases}
\Delta C_{k}\cos{((k+1)\theta)}+\Delta D_{k}\sin{((k+1)\theta)}=0,\\
-\Delta C_{k}\sin{((k+1)\theta)}+\Delta D_{k}\cos{((k+1)\theta)}=\widehat{m}_{k},
\end{cases}	
\end{equation}
where $\Delta C_{k}=C_{k+1}-C_{k}$ and $\Delta D_{k}=D_{k+1}-D_{k}$. In particular, the Casorati determinant corresponding to the coefficient matrix of the linear system \eqref{equ:DD} is $1$ as follows
\begin{equation}
\left|
\begin{array}{ccc}
 \cos{((k+1)\theta)} & \sin{((k+1)\theta)}\\
 -\sin{((k+1)\theta)} & \cos{((k+1)\theta)}
\end{array}
\right|=1.
\end{equation}
Hence, \eqref{equ:DD} admits a unique solution
\begin{equation}
	\Delta C_{k}=-\widehat{m}_{k}\sin{((k+1)\theta)}\quad\text{and}\quad\Delta D_{k}=\widehat{m}_{k}\cos{((k+1)\theta)}.
\end{equation}
From the definitions of $\Delta C_{k}$ and $\Delta D_{k}$, we have
\begin{equation}
	C_{k}=C_{0}-\sum_{j=0}^{k-1}\widehat{m}_{j}\sin{((j+1)\theta)}\quad\text{and}\quad D_{k}=D_{0}+\sum_{j=0}^{k-1}\widehat{m}_{j}\cos{((j+1)\theta)}.
\end{equation}
It is customary to set $C_{0}=D_{0}=0.$ Therefore, the particular solution of non-homogeneous term
\begin{align}\label{equ:particular}
	\begin{aligned}
	u_{k}^{(p)}&=-\cos{(k\theta)}\sum_{j=0}^{k-1}\widehat{m}_{j}\sin{((j+1)\theta)}+\sin{(k\theta)}\sum_{j=0}^{k-1}\widehat{m}_{j}\cos{((j+1)\theta)}\\
	&=\sum_{j=0}^{k-1}\widehat{m}_{j}\sin{((k-j-1)\theta)}.
	\end{aligned}
\end{align}

From \eqref{equ:general} and \eqref{equ:particular}, we have the solution  \eqref{equ:solution} as follows
\begin{equation}\label{solution}
	u_{k}=c_{1}\sin{(k\theta)}+c_{2}\cos{(k\theta)}+\sum_{j=0}^{k-1}\widehat{m}_{j}\sin{((k-j-1)\theta)}.
\end{equation}
It follows from the boundedness of $h_{k}$ in \eqref{equ:hk} and the definition of $\widehat{m}_{k}$ in \eqref{equ:mk} that $\widehat{m}_{k}$ is exponentially decaying overall and satisfies $\sum_{k=0}^{\infty}|\widehat{m}_{k}|<\infty$, which implies that $\widehat{m}_{k}$ is absolutely integrable. Define 
\begin{equation}\label{equ:S}
	S=\sum_{k=0}^{\infty}\widehat{m}_{k}e^{-i(k+1)\theta}=e^{-i\theta}\sum_{k=0}^{\infty}\widehat{m}_{k}e^{-ik\theta}.
\end{equation}
For convenience, we denote 
\begin{equation}
	S=|S|e^{iArg(S)},
\end{equation}
where $|S|$ represents the modulus of $S$, and $Arg(S)\in(-\pi,\pi]$ denotes the argument of $S$. We consider the asymptotic behavior of the particular solution. From \eqref{equ:particular}, we have
\begin{align*}
\begin{aligned}
u_{k}^{(p)}&=\sum_{j=0}^{k-1}\widehat{m}_{j}\sin{((k-j-1)\theta)}\\
&=\Im\big(\sum_{j=0}^{k-1}\widehat{m}_{j}e^{i(k-j-1)\theta}\big)\\
&=\Im\big(e^{i(k-1)\theta}\sum_{j=0}^{k-1}\widehat{m}_{j}e^{-ij\theta}\big),
\end{aligned}
\end{align*}
where $\Im$ denotes the imaginary part of a complex number. When $k$ is a large number such that $\sum_{j=0}^{k-1}\widehat{m}_{j}e^{-ij\theta}\rightarrow\sum_{j=0}^{\infty}\widehat{m}_{j}e^{-ij\theta}=e^{i\theta}S$ from \eqref{equ:S}. Then, we have 
\begin{equation*}
	e^{i(k-1)\theta}\sum_{j=0}^{k-1}\widehat{m}_{j}e^{-ij\theta}\rightarrow e^{i(k-1)\theta} \sum_{j=0}^{\infty}\widehat{m}_{j}e^{-ij\theta}=e^{ik\theta}S=|S|e^{i(k\theta+Arg(S))}.
\end{equation*}
Therefore, we have
\begin{equation}
	u_{k}^{(p)}=|S|\sin{(k\theta+Arg(S))}+\epsilon_{k},
\end{equation}
where $\epsilon_{k}\rightarrow0$ as $k\rightarrow\infty$. Thus, the solution to equation \eqref{equ:secondDF2} can be expressed as 
\begin{align}
\begin{aligned}
u_{k}&=c_{1}\sin{(k\theta)}+c_{2}\cos{(k\theta)}+|S|\sin{(k\theta+Arg(S))}+\epsilon_{k}\\
&=A\cos{(k\theta)}+B\sin{(k\theta)}+\epsilon_{k},
\end{aligned}
\end{align}
where $A=c_{1}+|S|\sin{(Arg(S))}$ and $B=c_{2}+|S|\cos{(Arg(S))}$. In the general case, $A$ and $B$ are not both $0$. Hence, for sufficiently large $k$, the behavior of $u_{k}$ is dominated by $A\cos{(k\theta)}+B\sin{(k\theta)}$ and can be uniformly expressed in the following form
\begin{align}\label{equ:uk}
	\begin{aligned}
	u_{k}=\sqrt{A^2+B^2}\cos{(\phi+k\theta)},
	\end{aligned}
\end{align}
where $\phi=\arccos{(\frac{A}{\sqrt{A^2+B^2}})}$ and $\theta=\arccos{(\frac{1}{\sqrt{8}})}$. Clearly, Equation \eqref{equ:uk} is a Chebyshev-like polynomial. According to the results in \cite[Theorem 1.3.3]{Agarwal2000Oscillation}, we know that $u_{k}$ is oscillatory. Therefore, $u_{k}$ alternates in sign infinitely. That is, $u_{k}$ oscillates around $0$ in general.

Based on transformations \eqref{equ:log} and \eqref{equ:ukk}, in the asymptotic sense, the value of $r_{k}$ can be expressed as 
\begin{equation}\label{equ:rk}
	r_{k}=e^{\sqrt{A^2+B^2}(\sqrt{2})^k\cos{(\phi+k\theta)}}.
\end{equation}
Since the exponential function is monotonically increasing and $e^{0}=1$, it follows that $r_{k}$ oscillates around $1$ with an increasing amplitude. However, the oscillations of $r_{k}$ become irregular due to the influence of the initial iterates, as well as the condition number and eigenvalues distribution of the Hessian matrix, and such information is embedded within $A$ and $B$. Nevertheless, as long as $r_{k}$ shows such oscillatory behavior with increasing amplitude overall, the desired result can be derived, as shown in the next section. 

\subsection{Convergence analysis for any-dimensional quadratics}\label{sec:2.1}
In this section, we will employ the results obtained in the previous section to establish the $R$-linear convergence of the RBB method.
\begin{lemma}\label{lemma:lem1}
	Assume that $r_{k}>0$ for some $k\ge1$. Then there exists an integer $N\ge 1$ such that $$\frac{r_{k+N}}{r_{k}}\le1.$$ 
\end{lemma}
\begin{proof}
Since the sequence $\{r_{k}\}$ oscillates around $1$, it possesses infinitely many local minima, which are denoted as $\{l_{j}\}, (j=1,2,\ldots)$. According to the asymptotic expression \eqref{equ:rk} of $r_{k}$, the amplitude of $r_{k}$ shows an overall increasing trend, with an asymptotic growth of order $\sqrt{2}$. This indicates that the degree to which the maxima or minima of the sequence $\{r_{k}\}$ deviate from $1$ is asymptotically increasing, and therefore the sequence of local minima exhibits an overall decreasing trend. By the definition \eqref{equ:nepsilon}, $r_{k}$ is bounded below by $0$. Hence, we have 
\begin{equation}\label{equ:liminf}
	\liminf_{j\to\infty} \{l_{j}\}=0.
\end{equation}
For a given $r_{k}>0$, due to the persistence of the oscillation and the limitation \eqref{equ:liminf}, there must exist some local minimum $l_{j^{'}}\le r_{k}$ after the index $k$. Hence, we can choose the index $k^{'}$ corresponding to $l_{j^{'}}$ and let $N=k^{'}-k$. Then we have  $r_{k+N}=l_{j^{'}}\le r_{k}$. This completes the proof.  
\end{proof}

\begin{remark}
	Generally, when the algorithm converges, $r_{k}$ tend to or equals zero, implying that $\|\g_{k}^{1}\|_{2}=0$. Consequently, from \eqref{equ:nnRBB}, we obtain $\alpha_{k+1}^{RBB}=1$, which finally eliminates the gradient component corresponding to the smallest eigenvalue of Hessian matrix, leading to $\|\g_{k+2}\|_2=0$.
\end{remark}

\begin{theorem}
	Let $f(\x)$ be a strictly convex quadratic function. Let $\{\x_{k}\}$ be the sequence generated by the RBB method. Then, either $\g_{k}=\xo$ for some finite $k$, or the sequence $\{\|\g_{k}\|_{2}\}$ converges to zero $R$-linearly.
\end{theorem}
\begin{proof}
	We only need to consider the case that $\g_{k}\neq\xo$ for all $k\ge1$. 
	From Lemma \ref{lemma:lem1} and \eqref{prod}, we know that there exists an integer $N\ge1$ such that  
	\begin{align}\label{equ:A}
	\begin{split}
	\frac{\|\g_{k+N}\|_{2}^{2}}{\|\g_{k}\|_{2}^{2}}&=\Big(\prod_{j=0}^{N-1}\xi_{k+j}^2\Big)\frac{r_{k+N}+1}{r_{k}+1}\\
	&\le\prod_{j=0}^{N-1}\xi_{k+j}^2,
	\end{split}
	\end{align}
	where each $\xi_{k+j}\in(0,1-\frac{1}{\lambda})$.  Then we have $\|\g_{k+N}\|_{2}\le \theta_{N}\|\g_{k}\|_{2}$, where $\theta_{N}=(1-\frac{1}{\lambda})^N$. Therefore, the sequence $\{\|\g_{k}\|_{2}\}$ converges to zero $R$-linearly. This completes our proof.	
\end{proof}

\subsubsection{Several special cases}\label{sec:SpecialCase}
In this subsection, we explore several special cases of $\alpha_{k}^{RBB}$ when $n=2$. Let
\begin{equation*}\label{equ:AA}
A=\begin{bmatrix}
\lambda & 0\\
0 & 1
\end{bmatrix},
\end{equation*}
with $\lambda>1$. In this case, we have
\begin{equation}\label{case2}
	\eta_{k-1}=\bar{\eta}_{k-1}=\eta_{k}=\bar{\eta}_{k}=1.
\end{equation} 
From \eqref{equ:nnRBB} and \eqref{case2}, we have 
\begin{align}\label{equ:2RBB}
\begin{split}
\alpha_{k}^{RBB}
&=\frac{1+\lambda r_{k-1}+\tau_{k}(1+\lambda^{2}r_{k-1})}{1+r_{k-1}+\tau_{k}(1+\lambda r_{k-1})}.
\end{split}
\end{align}
From \eqref{equ:betak}, \eqref{equ:xik}, and \eqref{case2}, we have  
\begin{align}\label{hghg}
\begin{split}
\bar{\beta}_{k}&=\frac{(\lambda-1)^2(1+\tau_{k})^2}{\big[1+r_{k-1}\lambda+\tau_{k}(1+r_{k-1})\lambda^2\big]^2}
\end{split}
\end{align}
\begin{align}\label{ghhg}
\begin{split}
\bar{\xi}_{k}=\frac{(\lambda-1) r_{k-1}+\tau_{k}(\lambda^2r_{k-1}-\lambda r_{k-1})}{1+\lambda r_{k-1}+\tau_{k}(\lambda^2 r_{k-1}+1)}\in(0,1-\frac{1}{\lambda}).
\end{split}
\end{align}
From \eqref{equ:nrecurion} and \eqref{case2}, we have 
\begin{equation}\label{equ:2dim}
\frac{(r_{k-1})^{2}r_{k+1}}{r_{k}}=\bar{h}_{k-1},
\end{equation}
where
\begin{equation}\label{lambar}
\bar{h}_{k-1}=\frac{(1+\tau_{k})^2}{(1+\tau_{k}\lambda)^2}\in(\frac{1}{\lambda^2},1].
\end{equation}

When $\tau_{k}=0$, i.e., $\alpha_{k}^{RBB}=\alpha_{k}^{BB1}$. In this case, $\bar{h}_{k-1}=1$,
\begin{equation}\label{equ:BB1}
\frac{(r_{k-1})^2r_{k+1}}{r_{k}}=1.
\end{equation}
Following the analysis in Section \ref{sec:PS}, the solution to equation \eqref{equ:BB1} can be expressed as 
\begin{equation}\label{equ:BB1solu}
	r_{k}=e^{2\sqrt{a^2+b^2}(\sqrt{2})^k\cos\big(\phi+k\arccos(\frac{1}{\sqrt{8}})\big)}
\end{equation}
where $\phi=\arccos(\frac{a}{\sqrt{a^2+b^2}})$, $a$ and $b$ depend on the initial iterates. Compared with \eqref{equ:rk}, here we restrict the analysis to the two-dimensional case. Consequently, factors such as the distribution of eigenvalues of the Hessian matrix no longer affect the non-homogeneous term, yielding a simpler result. Nevertheless, the fundamental structures of \eqref{equ:rk} and \eqref{equ:BB1solu} remain consistent, and thus their properties are essentially the same in general. To intuitively demonstrate the oscillatory behavior of $r_{k}$, Figure \ref{fig:arctan7} displays the variation of the function $\cos\big(k\arccos(\frac{1}{\sqrt{8}})\big)$ over the interval $k\in[1, 100]$. The numeric label adjacent to the small boxes in the figure indicate the corresponding iteration number $k$, which facilitates the examination of properties such as the oscillation period.

\begin{figure}[!h]
	\centering
	\subfigure{
		\includegraphics[width=0.8\textwidth]{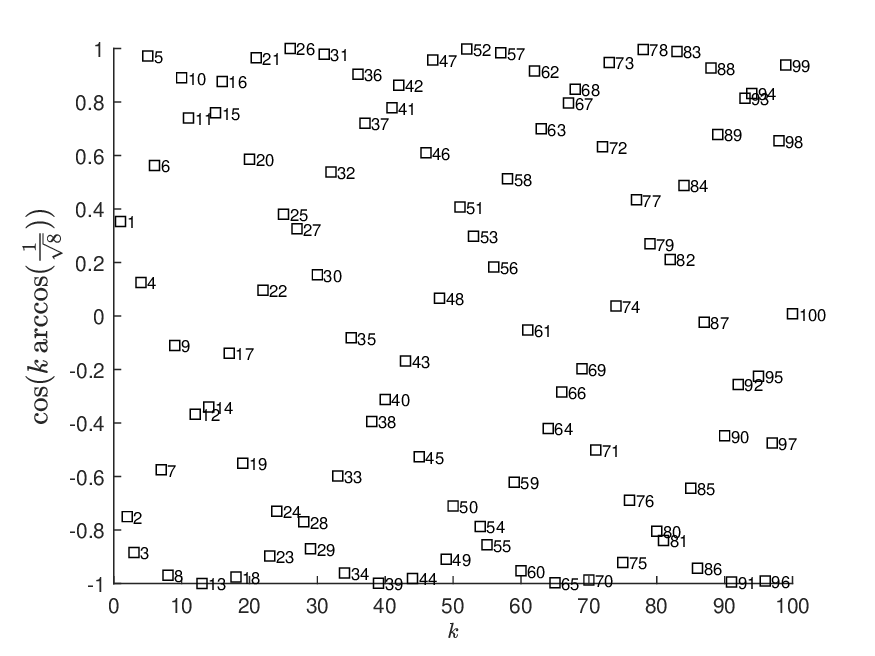}}
	\caption{\textit{The figure of the function $\cos(k\arccos(\frac{1}{\sqrt{8}}))$ for $k$ ranging from $1$ to $100$.}}	
	\label{fig:arctan7}
\end{figure}

Here, $\cos\big(k\arccos(\frac{1}{\sqrt{8}})\big)$ corresponds to the expression of the Chebyshev polynomial of the first kind evaluated at $x=\frac{1}{\sqrt{8}}$. The results in Figure \ref{fig:arctan7} clearly demonstrate that  $\cos\bigl(k\arccos(\frac{1}{\sqrt{8}})\bigr)$
oscillates around zero. Furthermore, the sign of $\cos\bigl(k\arccos(\frac{1}{\sqrt{8}})\bigr)$ changes approximately every three iterations. Consequently, the corresponding $r_{k}$ oscillates about $1$, which leads to a non-monotonic decrease in $\|\g_{k}\|_{2}$.


When $\tau_{k} \in [0,\infty)$, which corresponds to $\alpha_{k}^{RBB} \in [\alpha_{k}^{BB1}, \alpha_{k}^{BB2})$, we obtain from \eqref{equ:2dim} that $r_{k}$ and $\bar{h}_{k}$ satisfy the following properties
\begin{align}\label{equ:recursion}
\begin{split}
&(r_{k})^{2}r_{k+1}r_{k+3}=\bar{h}_{k}\bar{h}_{k+1},\\
&(r_{k})^{2}r_{k+1}(r_{k+2})^2r_{k+4}=\bar{h}_{k}\bar{h}_{k+1}\bar{h}_{k+2},\\
&(r_{k})^{2}r_{k+1}(r_{k+2})^2(r_{k+3})^2r_{k+5}=\bar{h}_{k}\bar{h}_{k+1}\bar{h}_{k+2}\bar{h}_{k+3},\\
&\ldots,\\
&(r_{k})^{2}r_{k+1}(r_{k+2})^2(r_{k+3})^2\cdots(r_{k+N})^2r_{k+N+2}=\bar{h}_{k}\bar{h}_{k+1}\bar{h}_{k+2}\bar{h}_{k+3}\cdots\bar{h}_{k+N}.\\
\end{split}
\end{align}
From an intuitive perspective, properties \eqref{equ:recursion} indicate that the sequence ${r_k}$ in the RBB method (for $\tau_k \in [0,\infty)$) exhibits weaker oscillation compared with the BB1 method. Based on these properties, it can also be shown that the sequence $\{\|\g_k\|_2\}$ converges to zero with an $R$-linear rate.

\begin{lemma}\label{lemma:lem111}
	Assume that $r_{k}>0$ for some $k\ge 1$. Then there exists an integer $N\ge 1$ such that $$\frac{r_{k+N}}{r_{k}}\le1.$$ 
\end{lemma}
\begin{proof}
	Assume, for the purpose of deriving a contradiction, that for all $N\ge1$,
	\begin{equation}\label{equ:aassume}
	\frac{r_{k+N}}{r_{k}}>1.
	\end{equation}
	It follows from  \eqref{equ:aassume} that 
	\begin{align}\label{equ:ineq}
	\begin{split}
	&(r_{k})^{2}r_{k+1}(r_{k+2})^2(r_{k+3})^2\cdots(r_{k+N})^2r_{k+N+2}\\
	&>(r_{k})^2r_{k}(r_{k})^2(r_{k})^2\cdots(r_{k})^2r_{k}=(r_{k})^{2N+2}.
	\end{split}
	\end{align}
	From \eqref{lambar} and \eqref{equ:recursion}, we know that 
	\begin{equation}\label{equ:eq}
	(r_{k})^{2}r_{k+1}(r_{k+2})^2(r_{k+3})^2\cdots(r_{k+N})^2r_{k+N+2}=\bar{h}_{k}\bar{h}_{k+1}\cdots\bar{h}_{k+N}\rightarrow 0, 
	\end{equation}
	as $N\rightarrow\infty$. From \eqref{equ:ineq} and \eqref{equ:eq}, our assumption \eqref{equ:aassume} must be false, which completes the proof.
\end{proof}

\begin{theorem}\label{lemma:lem11}
	Let $f(\x)$ be a strictly convex quadratic function. Let $\{\x_{k}\}$ be the sequence generated by the RBB method. Then, either $\g_{k}=\xo$ for some finite $k$, or the sequence $\{\|\g_{k}\|_{2}\}$ converges to zero $R$-linearly.
\end{theorem}
\begin{proof}
	From Lemma \ref{lemma:lem111} and \eqref{prod}, we know that there exists an integer $N\ge1$ such that  
	\begin{align}
	\begin{split}
	\frac{\|\g_{k+N}\|_{2}^{2}}{\|\g_{k}\|_{2}^{2}}&=\Big(\prod_{j=0}^{N-1}(\bar{\xi}_{k+j})^2\Big)\frac{r_{k+N}+1}{r_{k}+1}\\
	&\le\prod_{j=0}^{N-1}(\bar{\xi}_{k+j})^2,
	\end{split}
	\end{align}
	where each $\bar{\xi}_{k+j}\in(0,1)$. Therefore, the sequence $\{\|\g_{k}\|_{2}\}$ converges to zero $R$-linearly. This completes our proof.	
\end{proof}

\section{A unified perspective on gradient descent methods}\label{sec:discussions}

In this paper, the R-linear convergence of the RBB method for any-dimensional strictly convex quadratics is established by decomposing the gradient into two components and analyzing the behavior of the associated second-order difference equation defined by the ratio of their norms. In addition, the convergence results of this paper can be easily extended to the RBB-like methods with the following form
\begin{equation*}
\tilde{\alpha}_{k}^{RBB}=\frac{\mathbf{s}_{k-1}^{\T}\mathbf{y}_{k-1}+\tau_{k} \mathbf{y}_{k-1}^{\T} A^{m}\mathbf{s}_{k-1}}{\mathbf{s}_{k-1}^{\T}\mathbf{s}_{k-1}+\tau_{k} \mathbf{s}_{k-1}^{\T}A^{m} \mathbf{s}_{k-1}},
\end{equation*}
with $m\ge1$, which is the solution to the following regularized least squares problem 
\begin{equation*}
\min_{\alpha\in\mathbb{R}} \Big\{\Vert \alpha \s_{k-1}-\y_{k-1}\Vert_{2}^{2} + \tau_{k}\Vert \alpha A^{\frac{m}{2}}\s_{k-1}- A^{\frac{m}{2}}\y_{k-1}\Vert_{2}^{2} \Big\},
\end{equation*}
where $\tau_{k}\ge0$ is the regularization parameter. When $m=1$, we have  $\alpha_{k}^{RBB}=\tilde{\alpha}^{RBB}$.

It follows from \eqref{BB steps} that $\alpha_{k}^{BB1}$ can be regarded as resulting from a one-step delay of $\alpha_{k}^{SD}$. Moreover, inspired such a delayed strategy, \cite{Friedlander1998GradientMethodRetards} further proposed a class of generalized spectral gradient step-size methods with multi-step delay strategies, which take the following form
\begin{equation}\label{equ:delay}
\alpha_{k}=\frac{\g_{v(k)}^{\T}A^{\rho(k)}\g_{v(k)}}{\g_{v(k)}^{\T}A^{(\rho(k)-1)}\g_{v(k)}},
\end{equation}
where  
\begin{equation}\label{equ:vk}
v(k)\in\{k,k-1,\ldots,\max\{1,k-m\}\},
\end{equation}
\begin{equation*}
\rho(k)\in\{q_{1},\ldots,q_{m}\},\quad\text{for}\quad  k=0,1,2,\ldots,
\end{equation*}
where $m$ is a positive integer and $q_{j}\ge 1, j=1,2,\ldots,m$ is a real number. For a fixed delay step count $v(k)$, as the power $\rho(k)$ of the Hessian matrix $A$  increases (corresponding to the power method in numerical linear algebra), the term $\alpha_{k}$ acts as a scalar derived from the regularization term in the RBB method, aiming to obtain a larger scalar value similar to that in BB2, thereby enhancing the stability of the algorithm. Now, we consider the convergence of gradient descent methods with different delayed step-size strategies under the fixed condition $\rho(k)=1$. A natural inquiry arises: can the convergence analysis framework developed in this paper be applied to address this problem in a unified manner?

As indicated by \cite{Forsythe1968asymptoticdirectionsthes,Yuan2006newstepsizesteepest} and the analysis in Section \ref{sec:n-dimensional}, the behavior of the gradient descent method in high-dimensional problems is essentially identical to that in two-dimensional cases. In high-dimensional settings, the distribution of eigenvalues of the Hessian matrix may influence local behavioral characteristic of the gradient descent method, but it does not affect the overall convergence of the algorithm. Therefore, this section focuses solely on the two-dimensional case with Hessian $A$ from \eqref{equ:A} in Section \ref{sec:SpecialCase}. Under the condition that $\rho_{k}=1$, we denote $\alpha_{k}$ in \eqref{equ:delay} as $\alpha_{k}^{D}$ as follows
\begin{equation}\label{equ:delays}
\alpha_{k}^{D}=\frac{\g_{v(k)}^{\T}A\g_{v(k)}}{\g_{v(k)}^{\T}\g_{v(k)}},
\end{equation}
where $v(k)$ is from \eqref{equ:vk}.

\subsection{Convergence of multi-step delayed spectral gradient methods}
Following the analytical process outlined in Section \ref{sec:SpecialCase}, we know that the variable $r_{k}$ in the gradient descent method with step size $\frac{1}{\alpha_{k}^{D}}$ satisfies the following form 
\begin{equation}\label{equ:rD}
\frac{(r_{v(k)})^2r_{k+1}}{r_{k}}=1.
\end{equation}
In particular, when $v(k)=k-1$, Equation \eqref{equ:rD} reduces to \eqref{equ:BB1}, which corresponds to the iterative scheme of BB1. We now consider the general case when $v_{k}\le k-1$, with the constraint that $v_{k}$ is nonnegative. Let
\begin{equation*}
	v(k)=k-j,
\end{equation*}
where $j\in\{1,2,\ldots,k-1\}$. Following the procedure outlined in Section \ref{sec:PS}, we apply a logarithmic transformation to Equation \eqref{equ:rD} and obtain
\begin{equation}\label{equ:rvk}
	2\ln(r_{v(k)}) +\ln(r_{k+1})- \ln(r_{k})=0,
\end{equation} 
which is a high-order linear difference equation. Let 
\begin{equation}\label{equ:qk}
	q_{k}=\ln(r_{k}).
\end{equation}
Then the characteristic equation of \eqref{equ:rvk} is as follows
\begin{equation}\label{equ:character}
	q^{d}-q^{d-1}+2=0,
\end{equation}
where $d\ge 2$ is an integer.

\begin{theorem}\label{them:complex}
	The solution to Equation \eqref{equ:character} necessarily involves complex numbers, and all complex roots have moduli greater than $1$.
\end{theorem} 
\begin{proof}
	Equation \eqref{equ:character} is equivalent to 
	\begin{equation}\label{equ:qd}
		q^{d-1}(q-1)=-2.
	\end{equation}
Let $q$ be an arbitrary root of \eqref{equ:qd}. Then we have
\begin{equation}\label{equ:moduli}
	|q|^{d-1}|q-1|=2,
\end{equation} 
where $|\cdot|$ is the moduli of a complex number. We now assume that there exists a solution $q^{'}$ such that 
\begin{equation}\label{equ:le1}
	|q^{'}|\le 1.
\end{equation}
Then we have 
\begin{equation}\label{equ:qde}
	|q^{'}|^{d-1}\le 1.
\end{equation}
Substituting \eqref{equ:qde} into \eqref{equ:moduli} yields 
\begin{equation}\label{equ:condition1}
	|q^{'}-1|\ge 2.
\end{equation}	
Meanwhile, by the triangle inequality, we have
\begin{equation}\label{equ:condition2}
	|q^{'}-1|\le|q^{'}|+1\le 2.
\end{equation}
From \eqref{equ:moduli}, \eqref{equ:condition1} and \eqref{equ:condition2}, it must be that
\begin{equation*}
	|q^{'}-1|=2 \quad\text{and}\quad|q^{'}|^{d-1}=1,
\end{equation*}
hence $|q^{'}|=1$. On the unit circle, the only point satisfying $|q^{'}-1|=2$ is $q^{'}=-1$. Direct verification shows that $q^{'}=-1$ is a solution to Equation \eqref{equ:qd} when $d$ is odd, but it is not a solution when $d$ is even.

Therefore, when $d$ is even, all roots are complex number with moduli greater than $1$; when $d$ is odd, except for the real solution $q^{'}=-1$, all other solutions are complex numbers with moduli greater than $1$. We complete the proof.
	
\end{proof}

According to Theorem \ref{them:complex}, the solution of the characteristic equation \eqref{equ:character} can be expressed in the form
\begin{equation}\label{equ:cha}
	q=re^{i\theta},
\end{equation}
where $r>1$, $\theta\in(0,2\pi)$. Substituting \eqref{equ:cha} into \eqref{equ:character} yields
\begin{equation*}
	r^{d-1}e^{i(d-1)\theta}(re^{i\theta}-1)=-2.
\end{equation*}
This leads to the modulus equation
\begin{equation}\label{equ:rkk}
	r^{d-1}|re^{i\theta-1}|=2,
\end{equation}
where $|re^{i\theta-1}|=\sqrt{r^2+1-2r\cos(\theta)}$,
and
the phase equation
\begin{equation}\label{equ:theta}
	(d-1)\theta+\arg(re^{i\theta}-1)=\pi+2l\pi,\quad\text{for}\quad l\in\mathbb{Z},
\end{equation}
where $\arg(re^{i\theta}-1)=\arctan(\frac{r\sin(\theta)}{r\cos(\theta)-1})$. Moreover, it follows from \eqref{equ:cha} that $\theta\neq 0$; otherwise, this would contradict that there exists solution $q=-1$ when $d$ is odd.

Therefore, the solution of the difference equation \eqref{equ:rvk} have the following form  
\begin{equation}
	q_{k}=\widehat{c}r^{k}\cos(\widetilde{\phi}+k\arctan(\theta)),
\end{equation} 
where $\widehat{c}$ and $\widetilde{\phi}$ are determined initial iterates. Thus, we finally obtain the expression for $r_{k}$ as follows
\begin{equation}\label{equ:rkj}
	r_{k}=e^{\widehat{c}r^{k}\cos(\widetilde{\phi}+k\arctan(\theta))}.
\end{equation}
Note that Equations \eqref{equ:rkk} and \eqref{equ:theta} for $r$ and $\theta$, respectively, are complicated, and it is generally difficult to provide their analytical expressions. Nevertheless, the structure of \eqref{equ:rkj} is similar to that of \eqref{equ:BB1solu}. Since $\theta\neq 0$ and $r>1$, $r_{k}$ also exhibits oscillatory behavior, and the amplitude of the oscillation progressively increases. According to the conclusion of Lemma \ref{lemma:lem1}, it follows that the gradient descent method with a delayed multi-step strategy is also $R$-linearly convergent.

\subsection{Behavioral characteristics of the steepest descent method}
In this subsection, we examine the case where $v(k)=k$, corresponding to the behavioral characteristics of the steepest descent method, which does not employ a step-size delay strategy. By continue following the analysis procedure presented in Section \ref{sec:n-dimensional}, we have
\begin{equation}\label{equ:Drk}
	r_{k+1}=\frac{1}{r_{k}}.
\end{equation}

Compared with Equation \eqref{equ:rD}, Equation \eqref{equ:Drk} corresponds to a first-order difference equation. Thus, it can be observed that the step-size delay strategy increases the order of the difference equation associated with the gradient descent method, thereby accelerating convergence. According to expression \eqref{equ:Drk}, a notable characteristic of the SD method is that the generated sequence $r_{k}$ repeats every other step, i.e., it satisfies $r_{k+2}=r_{k}$. From relation \eqref{prod}, we have 
\begin{equation}
\frac{\|\g_{k+2}\|_{2}^{2}}{\|\g_{k}\|_{2}^{2}}=\xi_{k}^2\xi_{k+1}^2.
\end{equation}
Hence, it can be seen that the convergence of the SD method is strictly governed by $\prod\xi_{k}^{2}$, while the $\frac{r_{k+N}}{r_{k}}$ term in \eqref{prod} does not contribute to accelerating convergence. Conversely, for the gradient descent method with a delayed step-size strategy, taking BB1 as an example, the generated $\{r_{k}\}$ exhibits oscillatory behavior with increasing amplitude, which allows the $\frac{r_{k+N}}{r_{k}}$ in \eqref{prod} to exert a significant decreasing effect on the $\prod\xi_{k}^{2}$ term.

\section{Selecting regularization parameter}\label{sec:parameter}
In the RBB method, the selection of the regularization parameter $\tau_{k}$ remains a critical matter, as it governs the practical performance of the method. We now proceed to analyze how to choose an appropriate regularization parameter.

Recall from the convergence analysis in Section \ref{sec:n-dimensional} that the descent rates of $(\g_{k}^{(n)})^{2}$ and $\|\g_{k}^{1}\|_{2}$ are used to measure the convergence and stability (the opposite of oscillation) of the RBB method, respectively. These two factors collectively determine the overall performance of the algorithm. Based on the conclusions of Proposition \ref{prop:increasing}, Lemmas \ref{lemma:lem1} and \ref{lemma:lem111}, there are two ways to improve the performance of the BB method as follows
\begin{enumerate}
	\item Minimization of the parameter $\xi_{k}$;
	\item Minimization of the number of non-monotonic steps $N$.
\end{enumerate}

As shown in Proposition \ref{prop:increasing}, increasing the value of regularization parameter raises the value of $\xi_{k}$, thereby reducing the convergence rate of the algorithm. The non-monotone step count $N$ is closely related to the non-homogeneous term $h_{k-1}$. Generally, a larger $h_{k-1}$ causes greater interference with the corresponding homogeneous general solution, leading to a larger $N$. The following analysis shows that, under certain conditions, increasing the value of regularization parameter helps to a reduce the magnitude of $h_{k-1}$, which in turn contributes to lowering the value of $N$.

From \eqref{equ:hkkk}, for clarity, let 
\begin{equation}\label{equ:hkk}
h_{k-1}=\frac{N(\tau_{k})}{D(\tau_{k})},
\end{equation}
where 
\begin{align*}
N(\tau_{k})=\big(P+\tau_{k}C\big)^2+\big(Q+\tau_{k}E\big)^2F,\quad D(\tau_{k})=\big(G+\tau_{k}H\big)^2,
\end{align*} 
where
\begin{align}\label{parameter}
\begin{split}
F=\eta_{k}(\bar{\eta}_{k}-\eta_{k})\lambda^2\ge 0,\quad G=\eta_{k-1}\lambda-1>0,\quad H=\eta_{k-1}\lambda(\bar{\eta}_{k-1}\lambda-1)>0.
\end{split}
\end{align}
Let 
\begin{equation}\label{equ:tt}
	T=PC+FQE.
\end{equation}
Then we have the following conclusion.
\begin{proposition}\label{prop:CEF}
	If $\bar{\eta}_{k-1}\le\eta_{k}$. Then $C^2+E^2F<\eta_{k-1}\lambda T$.
\end{proposition}
\begin{proof}
	It follows from $\bar{\eta}_{k-1}\le\eta_{k}$ that 
	\begin{equation}\label{equ:pc}
		P<0\quad \text{and}\quad C<0
	\end{equation} 
	hold.
	From \eqref{parameter} and \eqref{equ:pc}, we have $T>0$.
	Let
	\begin{equation*}
	a=1-\eta_{k}\lambda<0,\quad b=r_{k-1}\lambda(\eta_{k-1}-\eta_{k})\le 0,\quad c=r_{k-1}\eta_{k-1}\lambda^2(\bar{\eta}_{k-1}-\eta_{k})<0.
	\end{equation*}
	Then 
	\begin{align}
	\begin{split}
	\eta_{k-1}\lambda T-(C^2+E^2F)&=\eta_{k-1}\lambda(PC+FQE)-(C^2+E^2F)\\
	&=\big[\eta_{k-1}\lambda(a+b)(a+c)-(a+c)^2\big]+\big[\eta_{k-1}\lambda FQE-E^2F\big]\\
	&=I_{1}+I_{2},
	\end{split}
	\end{align}
	where $I_{1}=(a+c)\big[\eta_{k-1}\lambda(a+b)-(a+c)\big],\quad I_{2}=FE(\eta_{k-1}\lambda Q-E)$.
	\begin{enumerate}
		\item[(I)] Prove $I_{1}>0$. To this end, we compute
		\begin{align}
		\begin{split}
		\eta_{k-1}\lambda b-c&=\eta_{k-1}r_{k-1}\lambda^2(\eta_{k-1}-\eta_{k})-r_{k-1}\eta_{k-1}\lambda^2(\bar{\eta}_{k-1}-\eta_{k})\\
		&=r_{k-1}\eta_{k-1}\lambda^2(\eta_{k-1}-\bar{\eta}_{k-1})\\
		&\le 0.
		\end{split}
		\end{align}
		Since $a+c<0$, we have $I_{1}=(a+c)\big[(\eta_{k-1}\lambda-1)a+\eta_{k-1}\lambda b-c\big]\ge 0$.
		
		\item[(II)] Prove $I_{2}>0$. Since $FE>0$, we consider only the sign of $\eta_{k-1}\lambda Q-E.$
		Then we have 
		\begin{align}
		\begin{split}
		\eta_{k-1}\lambda Q-E&=\eta_{k-1}\lambda(1+r_{k-1})-(1+r_{k-1}\eta_{k-1}\lambda)\\
		&=\eta_{k-1}\lambda-1\\
		&>0.
		\end{split}
		\end{align}
	\end{enumerate}	
	Combining $(I)$ and $(II)$, we have $I_{1}+I_{2}>0$. We complete the proof.	
	
\end{proof}

\begin{theorem}\label{the:mh}
	If $\bar{\eta}_{k-1}\le\eta_{k}$. Then the value of $h_{k-1}$ in \eqref{equ:nrecurion} is monotonically decreasing with respect to $\tau_{k}$.
\end{theorem}
\begin{proof}
From \eqref{equ:hkk}, we have 
\begin{equation*}
	\frac{d(h_{k-1})}{d(\tau_{k})}=\frac{N^{'}(\tau_{k})D(\tau_{k})-N(\tau_{k})D^{'}(\tau_{k})}{(D(\tau_{k}))^2}.
\end{equation*}
Since $(D(\tau_{k}))^2>0$, it suffices to consider the sign of the numerator. Calculation reveals that 
\begin{equation*}
	N^{'}(\tau_{k})D(\tau_{k})-N(\tau_{k})D^{'}(\tau_{k})=2(G+\tau_{k}H)S(\tau_{k}),
\end{equation*}	
where
\begin{equation*}
	S(\tau_{k})=\big[(P+\tau_{k}C)C+(Q+\tau_{k}E)EF\big](G+\tau_{k}H)-\big[(P+\tau_{k}C)^2+(Q+\tau_{k}E)^2F\big]H.
\end{equation*}	
Since $G>0$ and $H>0$, we consider only the sign of $S(\tau_{k})$, which can be rewritten in the following form
\begin{equation*}
	S(\tau_{k})=\big[GT-H(P^2+FQ^2)\big]+\big[G(C^2+E^2F)-HT\big]\tau_{k}.
\end{equation*}
	The following prove that $S(\tau_{k})<0$ for all $\tau_{k}\ge 0$.
\begin{enumerate}
	\item Slope analysis. It follows \eqref{parameter} and Proposition \ref{prop:CEF} that
	\begin{equation*}
		\frac{H}{G}>\eta_{k-1}\lambda\quad\text{and}\quad C^2+E^2F<\eta_{k-1}\lambda T.
	\end{equation*}
	Thus, we have 
	\begin{equation*}
		G(C^2+E^2F)-HT<(G\eta_{k-1}\lambda-H)T <0.
	\end{equation*}
	\item Intercept analysis. From Cauchy-Schwartz inequality, we know 
	\begin{equation}\label{equ:TT}
		T^2\le(P^2+FQ^2)(C^2+E^2F)<(P^2+FQ^2)\eta_{k-1}\lambda T.
	\end{equation} 
	From $T>0$ and \eqref{equ:TT}, we have 
	\begin{equation*}\label{equ:T}
		T<\eta_{k-1}\lambda(P^2+FQ^2).
	\end{equation*}
	Thus, we have
	\begin{align}
	\begin{split}
		GT&<G\eta_{k-1}\lambda(P^2+FQ^2)\\
		&<H(P^2+FQ^2).
	\end{split}	
	\end{align}
	So $S(0)<0$.
\end{enumerate}
Therefore, we have $h^{'}(\tau_{k})<0$ for all $\tau_{k}\ge 0$. We complete the proof.
	
\end{proof}

According to Theorem \ref{the:mh}, when $\bar{\eta}_{k-1}\le\eta_{k}$, choosing a larger regularization parameter $\tau_{k}$ conducive to reducing the value of $h_{k-1}$, that is, it decreases the input from the non-homogeneous term, thereby contributing to enhanced stability of the algorithm. Further, based on the definition of $\alpha_{k}^{RBB}$ in \eqref{equ:nnRBB}, it is high probability that $\bar{\eta}_{k-1}\le\eta_{k}$ implies  $\alpha_{k}^{RBB}\le\alpha_{k+1}^{RBB}$. Based on these analysis, we consider a regularization parameter scheme as follows
\begin{equation}\label{equ:parameter}
\tau_{k}=\mu_{k}\frac{\alpha_{k-1}^{RBB}}{\alpha_{k-2}^{RBB}},
\end{equation}
where $\mu_{k}>0$ is a regulatory factor. Ratio  $\frac{\alpha_{k-1}^{RBB}}{\alpha_{k-2}^{RBB}}$ describes the variability of the two step sizes in two consecutive iterations. This scheme suggests that the value of regularization parameter is proportional to the variability of step sizes in adjacent iterations. That is, under the conditions of high fluctuation, this scheme will generate a larger regularization parameter $\tau_{k}$, which means a larger penalty is imposed on the BB1 model, thereby resulting in a smaller step size and enhancing the algorithm's stability. Furthermore, we consider two choices of $\mu_{k}$ as follows
\begin{equation}\label{equ:muk}
	\mu_{k}=1\quad\text{and}\quad\mu_{k}=\alpha_{k-1}^{RBB},
\end{equation}
respectively. The former implies that $\tau_{k}$ only reflects the variability of the step size, while the latter further adjusts this variability using the step size scalar from the previous iteration. Since $\alpha_{k-1}^{RBB}$ contains local curvature information within the current neighborhood, the latter approach, in one sense, exhibits characteristics more akin to a damped Newton's method. 

Formally, the RBB method corresponding to the regularization parameter with regulatory factor $\mu_{k}=1$ is referred to as RBB1 method, while the one corresponding to the regularization parameter with regulatory factor $\mu_{k}=\alpha_{k-1}^{RBB}$ is referred to as RBB2 method. 

\section{Numerical experiment}\label{sec:NumerEx}
This experiment aims to examine the theoretical results presented in Section \ref{sec:n-dimensional} and to verify the effectiveness of the adaptive regularization parameter scheme \eqref{equ:parameter}. Therefore, in the comparative experiments\footnote{The experiments are conducted in the MATLAB environment.} only the RBB method incorporating regularization parameter \eqref{equ:parameter} is compared with the BB1 and BB2 methods.  We consider a quadratic objective function of the following form
\begin{equation}\label{pro:quaa}
f(\mathbf{x})=\frac{1}{2}(\mathbf{x}-\mathbf{x}_{*})^{\T}A(\mathbf{x}-\mathbf{x}_{*}),
\end{equation}
where $A\in\mathbb{R}^{n\times n}$ is symmetric positive definite (SPD), $\x_{*}\in\mathbb{R}^{n}$ is the unique optimal solution, where the type of optimization problem is determined by the choice of the Hessian matrix $A$.   

\subsection{A typical non-stochastic quadratic optimization problem}
In this class of problems \cite{DeAsmundis2014efficientgradientmethod}, the Hessian matrix $A$ is a diagonal matrix with the diagonal elements as follows
\begin{equation}\label{equ:non-sto}
a_{i}=10^{\frac{ncond}{n-1}(n-i)},\quad\text{for}\quad i=1,\ldots,n,
\end{equation}
where $ncond=\text{log}_{10}(\kappa(A))$. Since, for a given dimension $n$, the eigenvalues of the Hessian matrix $A$ are fixed and exhibit no clustering phenomenon, such problems are generally challenging to solve. However, the advantage is that once the problem dimension $n$, the initial iterate $\x_{1}$, and the optimal solution $\x_{*}$ are determined, the algorithmic results become fully reproducible, which facilitates the evaluation of the algorithm's performance.

We set $\x_{*}=\mathbf{1}$,   $\x_{1}=\mathbf{0}$, $n=5$, $\kappa(A)=10^{3}$ and use the stopping condition $
\Vert \g_{k}\Vert_{2}\le 10^{-20}\Vert \g_{1}\Vert_{2}$. Initial \text{step size} is $\frac{\g_{1}^{\T}\g_{1}}{\g_{1}^{\T}A\g_{1}}$. Figures \ref{figure:BB1} and \ref{figure:RBB} record the sequences $\{r_{k}\}$ and $\{\|\g_{k}\|_{2}\}$ generated during the search progress by the BB1 and RBB methods, respectively. Taking Figure \ref{figure:BB1} as an example for illustration, the label in the box indicated by the arrow in the figure represents the coordinates for the point. It should be noted that in the label for the gradient norm, we use $g_{k}$ to represent $\|\g_{k}\|_{2}$ for brevity.

\begin{figure}[h!]
	\centering
	\subfigure{
		\includegraphics[width=0.45\textwidth]{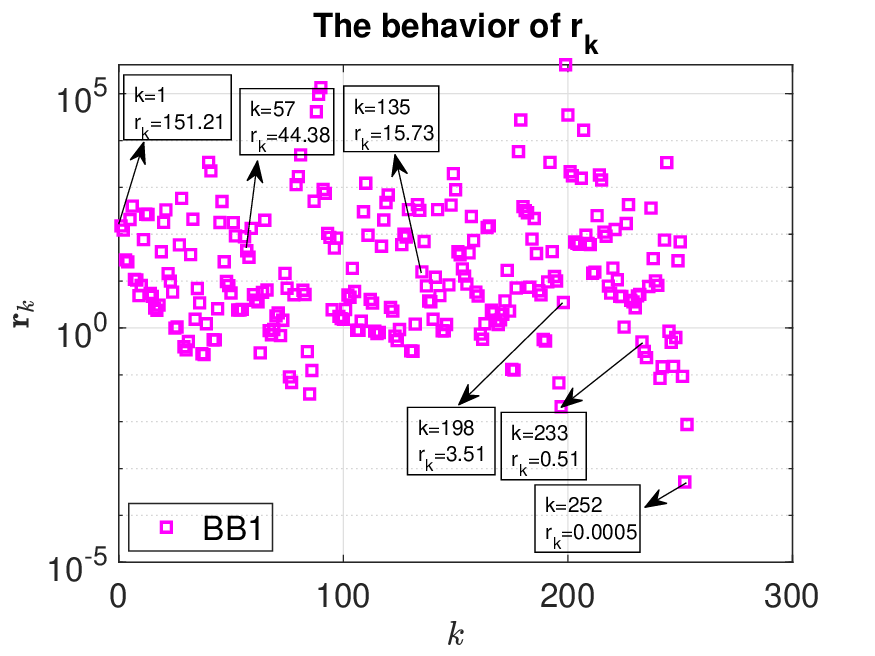}}\hspace{-4pt}
	\subfigure{
		\includegraphics[width=0.45\textwidth]{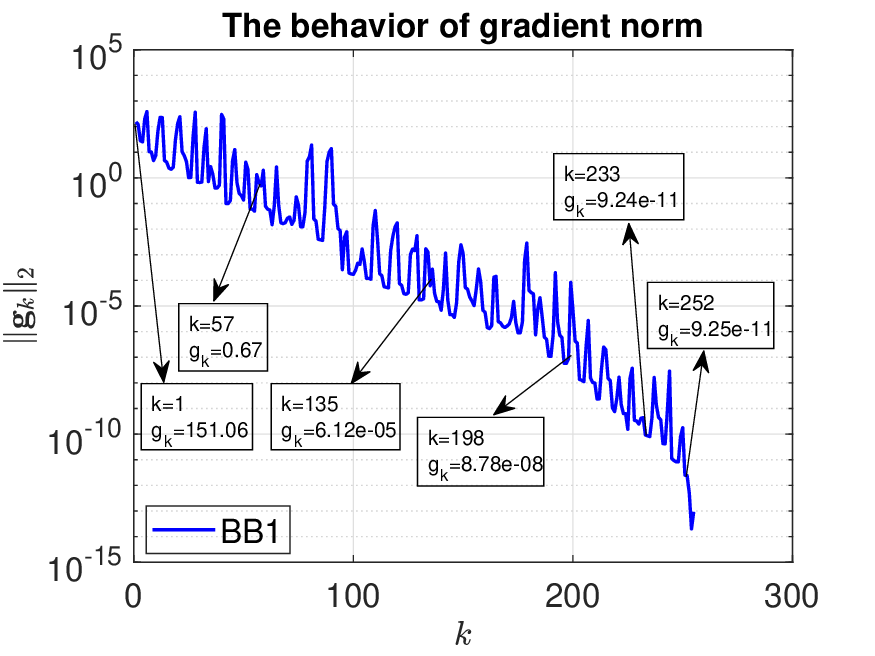}}\hspace{-4pt}
	\caption{\textit{Performance of BB1 method on the problem \eqref{pro:quaa}} with diagonal matrix \eqref{equ:non-sto}.}	
	\label{figure:BB1}	
\end{figure}

\begin{figure}[h!]
	\centering
	\subfigure{
		\includegraphics[width=0.45\textwidth]{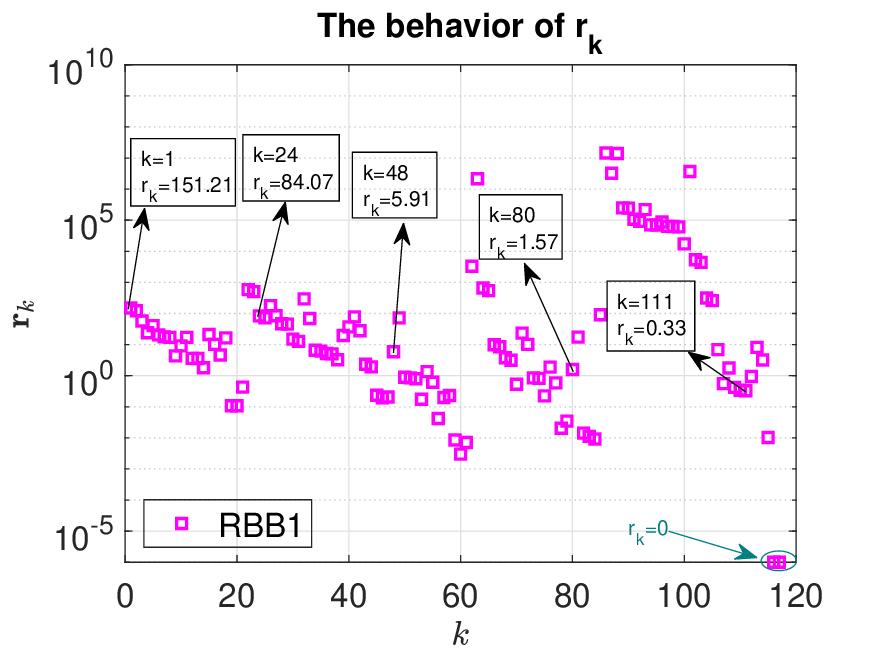}}\hspace{-4pt}
	\subfigure{
		\includegraphics[width=0.45\textwidth]{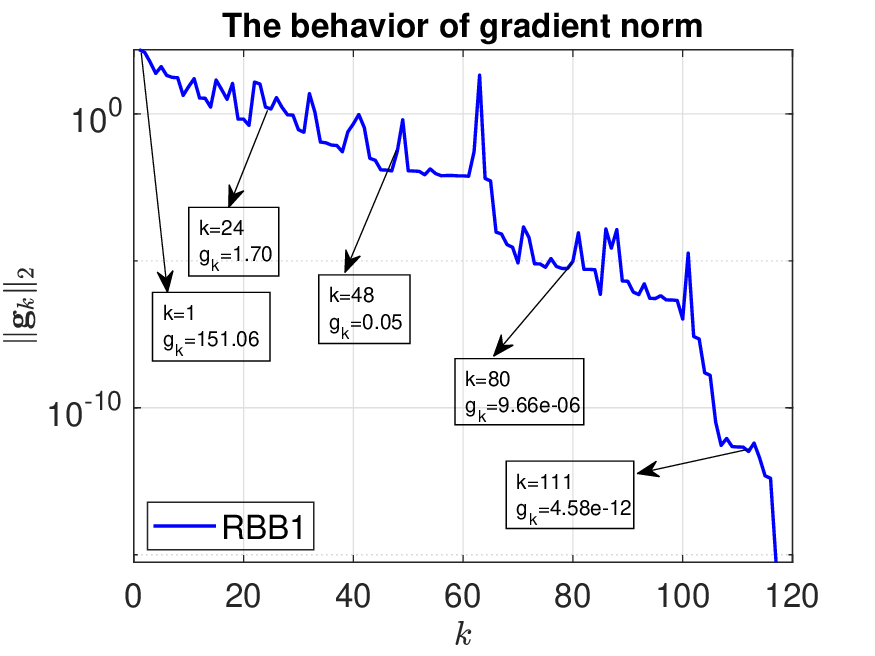}}
	\caption{\textit{Performance of RBB method on the problem \eqref{pro:quaa}} with diagonal matrix  \eqref{equ:non-sto}.}	
	\label{figure:RBB}	
\end{figure}

In this typical numerical example, the BB1 and RBB methods require $255$ and $117$ iterations, respectively, to meet the termination condition. These results demonstrate the effectiveness of the  regularization parameter scheme \eqref{equ:parameter}. Furthermore, a careful analysis of these two figures is worthwhile. 

Observing the $r_{k}$ metrics in both figures, they generally exhibit oscillations around $1$ with progressively increasing amplitudes. These numerical results validate the theoretical analysis presented in Section \ref{sec:n-dimensional}. Moreover, due to the negative feedback regulation effect of the regularization parameter on the non-homogeneous term $h_{k-1}$ in RBB method, the sequence $\{r_{k}\}$ in RBB shows greater stability compared to that in BB1.

By comparing the variation patterns of $\|\g_{k}\|_{2}$ and $r_{k}$, it can be seen that significant decreases in $\|\g_{k}\|_{2}$ originate from substantial declines in $r_{k}$. Notably, when $r_{k}\le1$, even if subsequent $r_{k}$ values decrease, the corresponding $\|\g_{k}\|_{2}$ does not show a significant reduction. A large-scale decrease in $\|\g_{k}\|_{2}$ only becomes possible after the rising phase of the $\cos(k\theta)$. Therefore, we are particularly concerned with the decreasing behavior of $r_{k}$ when $r_{k}>1$ (or $r_{k}\gg1$). Consequently, starting from $r_{1}$, when identifying the next index not exceeding $r_{1}$ in Figures \ref{figure:BB1} and \ref{figure:RBB}, we preferentially select indexes greater than $1$ for observation.

\section*{Declarations}

\begin{itemize}
	\item Funding: This work is supported by the National Science Foundation of China (NSFC: \# 12371099).
	\item Conflict of interest: Not applicable.
	\item Ethics approval and consent to participate: Not applicable.
	\item Data availability: All data that support the findings of this study are included within the article. 
	\item Materials availability: Not applicable.
	\item Code availability: Not applicable.  
\end{itemize}

\bibliographystyle{plain}

\end{document}